\documentclass[a4paper,11pt]{amsart}
\usepackage[foot]{amsaddr}

\makeatletter
\@namedef{subjclassname@2020}{\textup{2020} Mathematics Subject Classification}
\makeatother

\usepackage[utf8]{inputenc}
\usepackage{graphicx}
\usepackage{cite}
\usepackage[english]{babel}
\usepackage{todonotes}
\usepackage{a4wide}
\usepackage{amsmath}
\usepackage{amsfonts}
\usepackage{amssymb}
\usepackage{amsthm}
\usepackage[initials]{amsrefs}
\usepackage{physics}
\usepackage[draft=false, kerning=true]{microtype}

\newtheorem{lem}{Lemma}
\newtheorem{thm}[lem]{Theorem}

\numberwithin{equation}{section}
\theoremstyle{definition}

\theoremstyle{remark}

\usepackage{mathtools}
\DeclareMathOperator{\Li}{Li}

\DeclareMathOperator{\Mellin}{\mathcal{M}}

\DeclarePairedDelimiterX{\setm}[2]{\{}{\}}{#1\,\delimsize\vert\,\mathopen{}#2}
\let\abs\relax
\DeclarePairedDelimiter\abs{\lvert}{\rvert}%

\DeclarePairedDelimiter\floor{\lfloor}{\rfloor}
\DeclarePairedDelimiter\ceil{\lceil}{\rceil}

\allowdisplaybreaks[3]

\newcommand{\N}{\mathbb{N}}

\newcommand{\Oh}{\mathcal{O}}
\newcommand{\dt}{\mathrm{d}t}

\title[A Central Limit Theorem for Integer Partitions into Small Powers]{A Central Limit Theorem for\\ Integer Partitions into Small Powers}

\author[G. F. Lipnik]{Gabriel F. Lipnik}
\address[G. F. Lipnik]{
  Graz University of Technology, Institute of Analysis and Number
  Theory, 8010~Graz, Austria}
\email{math@gabriellipnik.at}

\author[M. G. Madritsch]{Manfred G. Madritsch}
\address[M. G. Madritsch]{
  Université de Lorraine, IECL, CNRS, F-54000 Nancy, France}
\email{manfred.madritsch@univ-lorraine.fr}

\author[R. F. Tichy]{Robert F. Tichy}
\address[R. F. Tichy]{
  Graz University of Technology, Institute of Analysis and Number
  Theory, 8010~Graz, Austria}
\email{tichy@tugraz.at}



\date{\today}

\usepackage{hyperref}

\hypersetup{%
pdftitle={A Central Limit Theorem for Integer Partitions into Small Powers},
pdfsubject={},
pdfauthor={Gabriel F. Lipnik, Manfred G. Madritsch, Robert F. Tichy},
pdfkeywords={integer partitions, partition function, central limit theorem, saddle-point method, Mellin transform}
hyperindex=true,
plainpages=false
}

\begin{document}

\begin{abstract}
  The study of the well-known partition function $p(n)$ counting the
  number of solutions to $n = a_{1} + \dots + a_{\ell}$ with integers
  $1 \leq a_{1} \leq \dots \leq a_{\ell}$ has a long history in
  combinatorics. In this paper, we study a variant, namely partitions
  of integers into
  \begin{equation*}
    n=\floor{a_1^\alpha}+\cdots+\floor{a_\ell^\alpha}
  \end{equation*}
  with $1\leq a_1< \cdots < a_\ell$ and some fixed
  $0 < \alpha < 1$. In particular, we prove a central limit theorem
  for the number of summands in such partitions, using the saddle-point method.
\end{abstract}

\maketitle

\section{Introduction}
For a positive integer $n$, let $f(n)$ denote the number of unordered
factorizations as products of integer factors greater than~$1$. Balasubramanian
and Luca~\cite{balasubramanian_luca2011:number_factorizations_integer}
considered the set
\[ \mathcal{F}(x)=\setm{m}{m\leq x,\, m=f(n)\text{ for some }n}.\]
In order to provide an upper bound for $\abs{\mathcal{F}(x)}$, they had to
analyse the number~$q(n)$ of partitions of~$n$ of the form
\[ n=\floor{\sqrt{a_1}}+\cdots+\floor{\sqrt{a_\ell}}\]
with integers $1\leq a_1\leq\cdots \leq a_\ell$, where $\floor{x}$ denotes the
integer part of $x$.

Chen and Li~\cite{chen_li2015:square_root_partition} proved a similar
result, and Luca and
Ralaivaosaona~\cite{luca_ralaivaosaona2016:explicit_bound_number}
refined the previous results to obtain the asymptotic formula
\[ q(n)\sim
Kn^{-8/9}\exp\biggl(\frac{6\zeta(3)^{1/3}}{4^{2/3}}n^{2/3}+\frac{\zeta(2)}{(4\zeta(3))^{1/3}}n^{1/3}\biggr),\]
where
\[K=\frac{(4\zeta(3))^{7/18}}{\pi
    A^2\sqrt{12}}\exp\biggl(\frac{4\zeta(3)-\zeta(2)^2}{24\zeta(3)}\biggr)\]
and $A$ is the Glaisher--Kinkelin constant.  Li and
Chen~\cite{li_chen2016:r_th_root, li_chen2018:r_th_root} extended the result to arbitrary
powers $0<\alpha<1$ not being of the form $\alpha=1/m$ for a positive
integer $m$. Finally, Li and Wu~\cite{li_wu2021:k_th_root} considered
the case of $\alpha=1/m$: They obtained a complete expansion along lines similar to the one in Tenenbaum, Wu and
Li~\cite{tenenbaum_wu_li2019:power_partitions_and} as well as in Debruyne
and
Tenenbaum~\cite{debruyne_tenenbaum2020:saddle_point_method_for_partitions}.

In the present paper, we take a different point of view. For fixed $0<\alpha<1$, we
consider the distribution of the length of restricted $\alpha$-partitions. A restricted
$\alpha$-partition of~$n$ is a representation of~$n$ of the form
\[ n=\floor{a_1^\alpha}+\cdots+\floor{a_\ell^\alpha}\] with
$1\leq a_1< a_2< \cdots < a_\ell$, and $\ell$ is called its length. In
contrast to the corresponding counting function~$\omega(n)$ in the
theory of primes, the problem of restricted partitions is rarely
discussed in the theory of integer partitions. A short mentioning can
be found in Erd\H{o}s and
Lehner~\cite{erdos_lehner1941:distribution_of_summands_in_partitions}. The
study of distinct components was introduced by
Wilf~\cite{wilf1983:problems_in_combinatorial_asymptotics}. First
considerations on the length of partitions of integers can be found in
Goh and
Schmutz~\cite{goh_schmutz1995:distinct_part_sizes_in_partitions}. Their
result was extended by
Schmutz~\cite{schmutz1994:part_sizes_of_partitions} to multivariate
cases under the Meinardus' scheme
(cf.~Meinardus~\cite{meinardus1954:meinardus_scheme}), and
Hwang~\cite{hwang2001:limit_theorems_number} provided an extended
version with weaker necessary conditions. Madritsch and
Wagner~\cite{madritsch_wagner2010:central_limit_theorem} finally
considered sets with digital restrictions, which lead to a Dirichlet
generating function having equidistant poles along a vertical line in
the complex plane. In the present paper, we use a similar method but
for the case of multiple poles on the real line.

\section{Main Result}
Let $0<\alpha<1$ be a fixed real number. We let $\Pi(n)$ denote the set of partitions of a
positive integer~$n$ into parts $\floor{a_j^{\alpha}}$ where each $a_j$
occurs at most once. These partitions are called (restricted) $\alpha$-partitions for
short.  Furthermore, let $q(n)=\abs{\Pi(n)}$ be the cardinality of the set
$\Pi(n)$.  Moreover, we let $\Pi(n,k)$ denote the subset of partitions $\Pi(n)$
whose length (number of summands) is $k$ and $q(n,k)=\abs{\Pi(n,k)}$ is its
cardinality.

In the present work, we consider the random variable~$\varpi_n$ counting the
number of summands in a random $\alpha$-partition of~$n$. The probability
distribution of $\omega_n$ is given by $\mathbb{P}(\omega_n = k)= q(n,k)/q(n)$. In
order to obtain a central limit theorem for~$\varpi_n$, we have to carefully
analyse the associated bivariate generating function $Q(z,u)$, which
is given by
\[ Q(z,u)=1+\sum_{n\geq 1}\sum_{k\geq 1} q(n,k)u^kz^n.\]
Furthermore, for a fixed integer $k \geq 1$, we let~$g(k)$ denote the number of integers $n \geq 1$ satisfying
$\floor{n^\alpha}=k$, i.e.,
\[
  g(k) \coloneqq \bigl\lceil (k+1)^{\beta}\bigr\rceil-\bigl\lceil k^{\beta}\bigr\rceil
\]
with $\beta\coloneqq 1/\alpha$. Then the following lemma holds.
\begin{lem}\label{lem:probability-generating-function}
  With the notation above, we have
  \begin{equation*}
    Q(z, u)
    =\prod_{k\geq1}\bigl(1+uz^k\bigr)^{g(k)}
    =1+\sum_{n\geq1}q(n)\mathbb{E}(u^{\varpi_n})z^n.
  \end{equation*}
  Furthermore, the mean length $\mu_n=\mathbb{E}(\varpi_n)$ and its variance
$\sigma_n^2=\mathbb{V}(\varpi_n)$ are given by
\[\mu_n=\frac{\left[z^n\right]Q_u(z,1)}{\left[z^n\right] Q(z,1)}
  \qq{and}
  \sigma_n^2=\frac{\left[z^n\right]Q_{uu}(z,1)}{\left[z^n\right] Q(z,1)}
  +\frac{\left[z^n\right]Q_u(z,1)}{\left[z^n\right] Q(z,1)}-\mu_n^2.\]
\end{lem}

\begin{proof}
By the definition of $g(k)$, we have $\floor{a_{j}^\alpha}=k$ for exactly $g(k)$ different integers~$a_{j}$. Thus, it follows that
\[Q(z,u)=\prod_{k\geq1}\bigl(1+uz^k\bigr)^{g(k)}.\]
Furthermore, it holds that
\begin{align*}
  Q(z,u) 
        &=1+\sum_{n\geq 1}\sum_{k\geq 1}q(n,k)u^kz^n\\
        &=1+\sum_{n\geq 1}q(n)\sum_{k\geq 1}\frac{q(n,k)u^k}{q(n)} z^n\\
        &=1+\sum_{n\geq 1}q(n)\mathbb{E}(u^{\varpi_n}) z^n.
\end{align*}

Next, we turn our attention to the mean length $\mu_n$. For the derivative of~$Q$ with respect to~$u$ we obtain 
\[ Q_u(z,u)=1+\sum_{n\geq 1}\sum_{k\geq 1}k q(n,k)u^{k-1}z^n.\]
Thus, it follows that
\[ \frac{[z^n]Q_u(z,1)}{[z^n]Q(z,1)}
=\frac{\sum_{k\geq 1}k q(n,k)}{\sum_{k\geq 1}q(n,k)}
=\frac{\sum_{k\geq 1}k q(n,k)}{q(n)}
=\sum_{k\geq 1}k \frac{q(n,k)}{q(n)}=\mathbb{E}(\varpi_n).\]
The equation for the variance follows similarly.
\end{proof}

Now we can state the main theorem of this work as follows.
\begin{thm}
  \label{thm:main}
  Let $0<\alpha<1$ and let $\varpi_n$ be the random variable counting
  the number of summands in a random partition of $n$ into
  $\alpha$-powers. Then $\varpi_n$ is asymptotically normally
  distributed with mean $\mathbb{E}(\varpi_n)\sim\mu_n$ and variance
  $\mathbb{V}(\varpi_n)\sim\sigma_n^2$, i.e.,
  \[
    \mathbb{P}\left(\frac{\varpi - \mu_n}{\sigma_n}< x\right)
    =\frac{1}{\sqrt{2\pi}}\int_{-\infty}^x e^{-t^2/2}\dd t+o(1),
  \]
  uniformly for all $x$ as $n\to\infty$. The mean~$\mu_n$ and the
  variance~$\sigma_n^2$ are given by
  \begin{equation}\label{mu}
    \mu_n = \sum_{k\geq1}\frac{g(k)}{e^{\eta k}+1} \sim c_{1}n^{1/(\alpha + 1)}
  \end{equation}
  and
  \begin{equation}\label{sigma}
    \sigma_n^2 = \sum_{k\geq1}\frac{g(k)e^{\eta k}}{(e^{\eta k}+1)^2}
      -\frac{\left(\sum_{k\geq1}\frac{g(k)ke^{\eta k}}{(e^{\eta k}+1)^2}\right)^2}
      {\sum_{k\geq1}\frac{g(k)k^2e^{\eta k}}{(e^{\eta k}+1)^2}}
      \sim c_{2}n^{1/(\alpha + 1)}
  \end{equation}
  where $\eta$ is the implicit solution of
  \[
    n=\sum_{k\geq1}\frac{k}{e^{\eta k}+1}.
  \]
  Explicit formul\ae{} for the occurring constants~$c_{1}$
  and~$c_{2}$ are given in~\eqref{eq:mu-explicit}
  and~\eqref{eq:sigma-explicit}, respectively.

  Finally, the tails of the distribution satisfy the
  exponential bounds
  \begin{equation*}
    \mathbb{P}\Bigl(\frac{\varpi - \mu_n}{\sigma_n}\geq x\Bigr)
    \leq\begin{cases}
      e^{-x^2/2}\left(1+\Oh((\log n)^{-3})\right) &\text{if }0\leq x\leq n^{1/(6\alpha+6)}/\log n,\\
      e^{-n^{1/(6\alpha+6)}x/2}\left(1+\Oh((\log n)^{-3})\right) &\text{if }x\geq n^{1/(6\alpha+6)}/\log n,
    \end{cases}
  \end{equation*}
  and the analogous inequalities also hold for $\mathbb{P}\bigl(\frac{\varpi - \mu_n}{\sigma_n}\leq -x\bigr)$.
\end{thm}

This result fits into the series of other results on partitions in integers of the form
$\lfloor k^\alpha\rfloor$ with $k\geq1$. In particular, if $\alpha=1$, then we
have the classic case of partitions and Erd\H{o}s and
Lehner~\cite{erdos_lehner1941:distribution_of_summands_in_partitions} showed that
$\mu_n\sim cn^{1/2}$. For  $\alpha>1$, not every integer has a
representation of the form $\lfloor k^\alpha\rfloor$ and there are gaps in the set $\setm{\floor{k^{\alpha}}}{k\in\N}$. This led Hwang~\cite{hwang2001:limit_theorems_number} to the result
$\mu_n\sim cn^{1/(\alpha+1)}$. Consequently, our result $\mu_n\sim c_1n^{1/(\alpha+1)}$ seems to be a natural extension
of these results.

One of the main difficulties of the case $0<\alpha<1$ lies in the special
structure of the function~$g(k)$. In particular, if $\alpha=1/m$ with $m\geq2$ being
an integer, then the parts of the partitions are $m$th roots and $g(k)$ is given by the polynomial
\[g(k)=(k+1)^m - k^m=\sum_{r=0}^{m-1}\binom{m}{r}k^r.\] However, in the general
case of $\alpha\not\in\mathbb{Q}$, we have an additional error term (see~\eqref{eq:g-k-with-error}) of which no explicit form
is known. This makes the analysis more involved.

Finally we want to mention that a local version of this central limit theorem is
the topic of a subsequent project. In particular, it seems that the above
mentioned error term of the function $g(k)$ needs further considerations in this
case.

\section{Main Idea, Outline and Tools for the Proof}
\label{sec:outline}

The proof of our main theorem consists of analytic and probabilistic
parts. In the analytic part, we use Mellin transform and the
saddle-point method. The probabilistic part is based on the use of
Curtiss' theorem for moment-generating functions. Before we give the
details of the proof, this section is dedicated to give an overview of
the main techniques and tools.

We first note that the central limit theorem for the random variable $\varpi_n$
is equivalent to the fact that the normalized moment-generating function
$M_n(t)=\mathbb{E}(e^{(\varpi_n-\mu_n)t/\sigma_n})$ tends to $e^{t^2/2}$ for
$t\to\infty$. Consequently, we will show this limit. Furthermore, the mentioned
tail estimates can be obtained by Chernov's bound.

Since $\mu_n/\sigma_n$ is constant, we have
$\mathbb{E}(e^{(\varpi_n-\mu_n)t/\sigma_n})=e^{-\mu_n/\sigma_n}\mathbb{E}(e^{\varpi_nt/\sigma_n})$.
We recall that by Lemma~\ref{lem:probability-generating-function}, the
probability-generating function $\mathbb{E}(u^{\varpi_n})$ is given by
\[\mathbb{E}(u^{\varpi_n})=\frac{[z^n]Q(z,u)}{[z^n]Q(z,1)}.\]
In other words, it is sufficient for our purpose to obtain the coefficient
of~$z^n$ in~$Q(z,u)$.  By Cauchy's integral formula, we derive
\[
  Q_n(u)\coloneqq [z^n]Q(z,u)
  =\frac1{2\pi i}\oint_{\left| z\right|=e^{-r}}
  z^{-n-1}Q(z,u)\dd{z}.
\]
A standard transformation yields for $r > 0$ that
\begin{equation}
  \label{integral:representation}
  Q_n(u) = \frac{e^{nr}}{2\pi}\int_{-\pi}^{\pi}\exp\bigl(int+f(r+it,u)\bigr)\dd{t}
\end{equation}
with
\[
  f(\tau,u)\coloneqq\log
  Q\left(e^{-\tau},u\right)=\sum_{k\geq 1}
  g(k)\log\bigl(1+ue^{-k\tau}\bigr).
\]

For the integral in~\eqref{integral:representation}, we use the well-known
saddle-point method, also known as the method of steepest decent. The
main application of this method is to obtain estimates for integrals of the form
\[\int_{-\pi}^\pi e^{g(r+it)}\mathrm{d}t\]
for some suitable function~$g$. We choose $t_n>0$ in order to split the
integral up into two parts, one near the positive real axis and the other one
for the rest, i.e.,
\begin{gather}\label{integral:split}
  \int_{-\pi}^\pi e^{g(r+it)}\dd{t}
  =\int_{\left| t\right|\leq t_n} e^{g(r+it)}\dd{t}
  +\int_{t_n<\left| t\right| \leq \pi} e^{g(r+it)}\dd{t}.
\end{gather}
  
For the second integral, we compare the contribution of the integrand with the
contribution from the real line, i.e., we estimate $e^{g(r+it)-\Re(g(r))}$. This will
contribute to the error term.

For the first integral in~\eqref{integral:split}, we use a third order Taylor
expansion of $g(r+it)$ around $t=0$, which is
\begin{equation}
  \label{eq:taylor}
  g(r+it)=g(r) + it
  g'(r)-\frac{t^2}2g''(r)+\Oh\Bigl(\sup_t\left|t^3g'''(r+it)\right|\Bigr).
\end{equation}
Now we
choose $r$ such that the first derivative $g'(r)$ vanishes. Then in the
integral we are left with
\[\int_{\left| t\right|<t_n} e^{g(r+it)}\dd{t}
=e^{g(r)}\int_{\left| t\right|<t_n} e^{-\frac{t^2}{2}g''(r)}\Bigl(1+\Oh\Bigl(\sup_t
\left|t^3g'''(r+it)\right|\Bigr)\Bigr)\dd{t}.\]

Now the integrand is the one of a Gaussian integral and so we add the
missing part. The Gaussian integral contributes to the main part, and we need to
analyse $g''(r)$ and $g'''(r)$ in order to show that all our
transformations and estimates are valid. For
more details on the saddle-point method, we refer the interested reader to Flajolet and
Sedgewick~\cite[Chapter~VIII]{flajolet_sedgewick2009:analytic_combinatorics}.

The estimates for $g''(r)$ and $g'''(r)$ are based on singular analysis using the well-known
Mellin transform.
The Mellin transform $h^{*}(s)$ of a function~$h$ is defined by
\[ h^*(s)=\Mellin[h;s]=\int_0^\infty h(t)t^{s-1}\dd{t}.\]
The most important property for our considerations is the so
called rescaling rule, which is given by
\[\Mellin\biggl[\sum_{k}\lambda_kh(\mu_kx);s\biggr]=\biggl(\sum_{k}\frac{\lambda_k}{\mu_k^s}\biggr)h^*(s);\]
see~\cite[Theorem 1]{flajolet_gourdon_dumas1995:mellin_transforms_and}.  This
provides a link between a generating function and its Dirichlet generating
function. For a detailed account on this integral transformation, we refer the
interested reader to the work of Flajolet, Gourdon and
Dumas~\cite{flajolet_gourdon_dumas1995:mellin_transforms_and} and to the work
of Flajolet, Grabner, Kirschenhofer, Prodinger and Tichy~\cite{flajolet_grabner_et_al1994:mellin_transforms}.

Let $\delta>0$. Throughout the rest of our paper we assume $\delta\leq u\leq
\delta^{-1}$ and by ``uniformly in $u$'' we always mean ``uniformly as
$\delta\leq u\leq \delta^{-1}$''. The following theorem concerning the Mellin
transform will be helpful multiple times in the proof of our main result.
\begin{thm}[{Converse Mapping~\cite[Theorem
4]{flajolet_gourdon_dumas1995:mellin_transforms_and}}]\label{fgd:thm4}
Let $f(x)$ be continuous in $(0,+\infty)$ with Mellin transform $f^*(s)$ having
a nonempty fundamental strip $\langle\alpha,\beta\rangle$.
Assume that $f^*(s)$ admits a meromorphic continuation to the strip $\langle
\gamma, \beta, \rangle$ for some $\gamma<\alpha$ with a finite number of poles
there, and is analytic on $\Re(s)=\gamma$. Assume also that there exists a real
number $\eta\in(\alpha,\beta)$ such that
\begin{gather}\label{fgd:eq23}
  f^*(s)=\Oh(\abs{s}^{-r})
\end{gather}
with $r > 1$ as $\abs{s}\to\infty$ in $\gamma\leq\Re(s)\leq\eta$. If $f^*(s)$ admits the
singular expansion
\begin{gather}\label{fgd:eq24}
  f^*\asymp \sum_{(\xi,k)\in A}d_{\xi,k}\frac{1}{(s-\xi)^k}
\end{gather}
for $s\in \langle\gamma,\alpha\rangle$, then an asymptotic expansion of $f(x)$
at $0$ is given by
\[f(x)=\sum_{(\xi,k)\in A}d_{\xi,k}\left(\frac{(-1)^{k-1}}{(k-1)!}x^{-\xi}(\log x)^k\right)+\Oh(x^{-\gamma}).\]
\end{thm}
Roughly speaking, converse mapping says that the asymptotic expansion of~$f$ is
determined by the poles of its Mellin transform.

In the present paper, we apply this converse mapping to products of the Gamma
function~$\Gamma(s)$, the Riemann zeta function~$\zeta(s)$ and the polylogarithm
$\Li_s(z)$, which are defined by
\[\Gamma(s)=\int_0^\infty e^{-x}x^{s-1}\dd{x},\quad
  \zeta(s)=\sum_{n\geq1} n^{-s}\qq{and} \Li_s(z)=\sum_{n\geq1}
  \frac{z^n}{n^s},\]
respectively. To apply converse mapping, we need to show
that~\eqref{fgd:eq23} as well as~\eqref{fgd:eq24} are both fulfilled for these
functions.  Stirling's formula yields for the Gamma function that
\[\abs{\Gamma(x+iy)}=\sqrt{2\pi}\abs{y}^{x-1/2}e^{-\pi\abs{y}/2}\left(1+\Oh_{a,b}(1/y)\right)\]
for $a\leq x\leq b$ and $\abs{y}\geq1$.
Furthermore, the Riemann zeta function satisfies
\[\zeta(x+iy)\ll_{a,b}1+\abs{y}^A\] for
suitable $A=A(a,b)$, $a\leq x\leq b$ and $\abs{y}\geq1$. For the polylogarithm we follow the
ideas of Flajolet and Sedgewick~\cite[VI.8]{flajolet_sedgewick2009:analytic_combinatorics}.
This is a good application of the converse mapping, so we want to reproduce it
here: First of all, let $w=-\log z$ and define the function
\[\Lambda(w):=\Li_{\alpha}(e^{-w})=\sum_{n\geq1}\frac{e^{-nw}}{n^\alpha}.\]
This is a harmonic sum and so we apply Mellin transform theory. The
Mellin transform of~$\Lambda(w)$ satisfies
\[
  \Lambda^*(s)
  =\int_{0}^\infty \sum_{n\geq1}\frac{e^{-nw}}{n^\alpha}w^{s-1}\mathrm{d}w
  =\sum_{n\geq1}\frac{1}{n^{\alpha+s}} \int_{0}^\infty e^{-v}v^{s-1}\mathrm{d}v
  =\zeta(s+\alpha)\Gamma(s)
\]
for $\Re(s)>\max(0,1-\alpha)$.
The Gamma function has simple poles in the negative integers and
$\zeta(s+\alpha)$ has a simple pole in $1-\alpha$. Thus, the application of converse mapping
(Theorem~\ref{fgd:thm4}) yields
\[
  \Li_\alpha(z)=\Gamma(1-\alpha)w^{\alpha-1}
    +\sum_{j\geq0}\frac{(-1)^j}{j!}\zeta(\alpha-j)w^j
  \]
  with
  \[w=\sum_{\ell\geq1}\frac{(1-z)^\ell}{\ell}.\]
Using these estimates, we obtain an asymptotic formula for $Q_n(u)$ of the form
\[
  Q_n(u)=\frac{e^{nr+f(u,r)}}{\sqrt{2\pi B}}\bigl(1+\Oh\bigl(r^{2\beta/7}\bigr)\bigr).
\]
Recall that
\[M_n(t)=\mathbb{E}(e^{(\varpi_n-\mu_n)t/\sigma_n})
=\exp\Bigl(-\frac{\mu_nt}{\sigma_n}\Bigr)
\mathbb{E}(e^{t\varpi_n/\sigma_n})
=\exp\Bigl(-\frac{\mu_nt}{\sigma_n}\Bigr)\frac{Q_n(e^{t/\sigma_n})}{Q_n(1)}.\] Using
implicit differentiation, we obtain a Taylor expansion for the moment-generating
function, which yields
\[
  M_n(t)  =\exp\Bigl(\frac{t^2}2+\Oh\bigl(n^{2\beta/7}\bigr)\Bigr),
\]
proving the central limit theorem for $\varpi_n$. Finally, we will use Chernov's bound
for the tail estimates.

\section{Proof of the Main Result}

To prove Theorem~\ref{thm:main}, we apply the method we have outlined in the
previous section. As indicated above, we choose $r=r(n,u)$ such that the first
derivative in~\eqref{eq:taylor} vanishes, i.e.,
\[\left.\frac{\partial \left(int+f(r+it,u)\right)}{\partial
      t}\right|_{t=0}
  =in-i\sum_{k\geq1}\frac{kg(k)}{u^{-1}e^{kr}+1}=0.
\]
Since the sum is decreasing in $r$, we see that this equation has a
unique solution, which is the saddle point.  The main value of the
integral in~\eqref{integral:representation} lies around the positive
real axis. We set $t_n=r^{1+3\beta/7}$ and split the integral into two
ranges, namely into
\[
  Q_n(u) = I_1+I_2
\]
where
\begin{equation*}
  I_1=\frac{e^{nr}}{2\pi}\int_{\abs{t}\leq t_n}\exp\bigl(int+f(r+it,u)\bigr)\dd{t}
\end{equation*}
and
\begin{equation*}
  I_2=\frac{e^{nr}}{2\pi}\int_{t_n<\abs{t}\leq \pi}\exp\bigl(int+f(r+it,u)\bigr)\dd{t}.
\end{equation*}

\subsection{Estimate of $I_1$}
\label{sec:I1}
We start our considerations with the central integral $I_1$ and show
the following lemma on its asymptotic behavior.
\begin{lem}
  Let $B^{2} = f_{\tau\tau}(r,u)$. Then we have
  \begin{equation*}
    I_{1} = \frac{e^{nr+f(r, u)}}{2\pi}\left(\int_{-t_{n}}^{t_{n}}
  \exp\left(-\frac{B^2}{2}t^2\right)\dd{t}\right)\Bigl(1+\Oh\bigl(r^{2\beta/7}\bigr)\Bigr)
\end{equation*}
with
\begin{equation*}
 \int_{-t_{n}}^{t_{n}}
 \exp\left(-\frac{B^2}{2}t^2\right)\dd{t} = \frac{\sqrt{2\pi}}{B}+\Oh\biggl(
    r^{-1-3\beta/7}B^{-2}\exp\biggl(-\frac{r^{-\beta/7}}{2}\biggr)\biggr)
  \end{equation*}
  uniformly in~$u$.
\end{lem}

Recall that by ``uniformly in $u$'' we always mean ``uniformly as $\delta\leq
u\leq \delta^{-1}$''.

\begin{proof}
  It holds that
\begin{equation}
  \label{eq:expansion-of-f}
  f(r+it,u)=f(r,u)+f_\tau(r,u)it-\frac{f_{\tau\tau}(r,u)}2t^2+\Oh\Bigl(t^3\sup_{0\leq
    t_0\leq t}\abs{f_{\tau\tau\tau}(r+it_0,u)}\Bigr).
\end{equation}
Let $m$ be the integer with $m<\beta=1/\alpha\leq m+1$. 
Since
\begin{align}
  g(k)&=\ceil[\big]{(k+1)^{\beta}} - \ceil[\big]{k^{\beta}}\notag\\
      &=(k+1)^\beta-k^\beta+\Oh(1)\label{eq:g-k-with-error}\\
      &=\sum_{\nu=1}^m\binom{\beta}{\nu}k^{\beta-\nu}+\Oh(1),\notag
\end{align}
we derive
\begin{equation}
  \label{f_as_sum}
  \begin{split}
    f(\tau,u)&=\sum_{k\geq1}g(k)\log{(1+ue^{-k\tau})}\\
    &=-\sum_{k\geq1}g(k)\sum_{\ell\geq1}\frac{(-u)^\ell}{\ell}e^{-k\ell
      \tau}\\
    &=-\sum_{k\geq1}\sum_{\nu=1}^m\binom{\beta}{\nu}k^{\beta-\nu}\sum_{\ell\geq1}\frac{(-u)^\ell}{\ell}e^{-k\ell
      \tau}
    +\Oh\Biggl(\sum_{k\geq1}\sum_{\ell\geq1}\frac{(-u)^\ell}{\ell}e^{-k\ell \tau}\Biggr)\\
    &=-\sum_{\nu=1}^m\binom{\beta}{\nu}\sum_{k\geq1}k^{\beta-\nu}\sum_{\ell\geq1}\frac{(-u)^\ell}{\ell}e^{-k\ell
      \tau}
    +\Oh\Biggl(\sum_{k\geq1}\sum_{\ell\geq1}\frac{(-u)^\ell}{\ell}e^{-k\ell
        \tau}\Biggr).
  \end{split}
\end{equation}
For $j\geq0$, we analogously obtain
\begin{multline}
  \label{derivatives_of_f_as_sum}
  \frac{\partial^j f}{\partial \tau^j}(\tau, u)
  =(-1)^{j+1}\sum_{\nu=1}^m\binom{\beta}{\nu}\sum_{k\geq1}k^{\beta-\nu+j}\sum_{\ell\geq1}(-u)^\ell\ell^{j-1}e^{-k\ell
     \tau}\\
     +\Oh\Biggl(\sum_{k\geq1}k^j\sum_{\ell\geq1}(-u)^\ell
       \ell^{j-1}e^{-k\ell \tau}\Biggr).
\end{multline}
All infinite sums in~\eqref{derivatives_of_f_as_sum} are of the form
\[
  h_{\gamma,j}(\tau,u)=\sum_{k\geq1}k^{\gamma+j}\sum_{\ell\geq1}(-u)^\ell\ell^{j-1}e^{-k\ell \tau}
\]
with $\gamma > 0$. Let $H_{\gamma,j}(s,u)$ denote the Mellin transform
of $h_{\gamma,j}(\tau,u)$ with respect to $\tau$, then
$H_{\gamma,j}(s,u)$ is given by
\begin{align*}
  H_{\gamma,j}(s,u)
  &=\zeta(s-\gamma-j)\Li_{s-j+1}(-u)\Gamma(s).
\end{align*}
The function
$H_{\gamma,j}(s,u)$ converges for $\Re s>\gamma+j+2$ and its only pole in the range $j+\frac12<\Re s\leq\gamma+j+2$ is the one
from $\zeta(s-\gamma-j)$ in $s=\gamma+j+1$. The Riemann zeta function
and the polylogarithm grow only polynomially, whereas the Gamma
function decreases exponentially on every vertical line in the complex
plane; see Section~\ref{sec:outline}. Thus, we can apply converse mapping (Theorem~\ref{fgd:thm4})
and obtain
\[h_{\gamma,j}(\tau,u)=\Li_{\gamma+2}(-u)\Gamma(\gamma+j+1)\tau^{-\gamma-j-1}+\Oh\bigl(\tau^{-j-\frac12}\bigr).\]
By plugging everything into~\eqref{derivatives_of_f_as_sum}, we obtain
\begin{equation}
  \label{eq:f-tau}
  \begin{aligned}
    \frac{\partial^jf}{\partial \tau^j}(\tau,u)
    &=(-1)^{j+1}\sum_{\nu=1}^m\binom{\beta}{\nu}h_{\beta-\nu,j}(\tau,u)
    +\Oh\left(h_{0,j}(\tau,u)\right)\\
    &=(-1)^{j+1}\sum_{\nu=1}^m\binom{\beta}{\nu}\Li_{\beta-\nu+2}(-u)\Gamma(\beta-\nu+j+1)\tau^{-\beta+\nu-j-1}
    +\Oh\left(\tau^{-j-1}\right).
  \end{aligned}
\end{equation}
For the saddle point~$n$, this results in
\begin{equation}\label{n_as_function_of_r}
  n = -f_{\tau}(r, u) = -\sum_{\nu = 1}^{m}\binom{\beta}{\nu}\Li_{\beta-\nu+2}(-u)\Gamma(\beta -\nu +2) r^{-\beta+\nu-2} + \Oh\bigl(r^{-2}\bigr),
\end{equation}
whereas the second derivative is given by
\begin{equation*}
  B^{2} = f_{\tau\tau}(r, u) = -\sum_{\nu = 1}^{m}\binom{\beta}{\nu}\Li_{\beta-\nu+2}(-u)\Gamma(\beta -\nu + 3) r^{-\beta+\nu-3} + \Oh\bigl(r^{-3}\bigr).
\end{equation*}
For the third derivative occurring in the error term in~\eqref{eq:expansion-of-f}, we have
\begin{align*}
  \frac{\partial^3 f(r+it, u)}{\partial t^3}
  &=-i\sum_{k\geq1}\frac{k^{3}g(k)u^{-1}e^{k(r+it)}(1-u^{-1}e^{k(r+it)})}
  {(u^{-1}e^{k(r+it)}+1)^{3}}.
\end{align*}
We estimate this expression
following~\cite{madritsch_wagner2010:central_limit_theorem}: Let $k_{0} = r^{-(1+c)}$
for some constant $c > 0$ and write $v = u^{-1}$ for short. We
split the sum into two parts, according to whether $k \leq k_{0}$ or not. 
For the sum over large~$k$, we obtain
\begin{multline}
  \label{eq:sum-over-large-k}
  \abs[\bigg]{-i\sum_{k > k_{0}}\frac{k^{3}g(k)ve^{k(r+it)}(1-ve^{k(r+it)})}
    {(ve^{k(r+it)}+1)^{3}}} \leq \sum_{k > k_{0}}
  \frac{k^{3}g(k)ve^{kr}(1+ve^{kr})}{\abs{ve^{k(r + it)} + 1}^{3}}\\
  \leq \sum_{k > k_{0}} \frac{k^{3}g(k)ve^{kr}(1 + ve^{kr})}{(ve^{kr} -
    1)^{3}} \ll \sum_{k > k_{0}} \frac{k^{3}g(k)}{e^{kr}} \ll r^{-\beta-3}.
\end{multline}
For the remaining sum we note that
\begin{equation*}
  \abs{1 + ve^{k(r + it)}} \geq (1 + ve^{kr})\cos\Bigl(\frac{kt}{2}\Bigr).
\end{equation*}
Therefore, we get
\begin{align*}
  \abs[\bigg]{-i\sum_{k\leq k_{0}}\frac{k^{3}g(k)u^{-1}e^{k(r+it)}(1-u^{-1}e^{k(r+it)})}
  {(u^{-1}e^{k(r+it)}+1)^{3}}} &\leq \sum_{k\leq
  k_{0}}\frac{k^{3}g(k)ve^{kr}(1 + e^{kr})}{\abs{ve^{k(r+it)} + 1}^{3}}\\
  &\leq \sum_{k\leq k_{0}} \frac{k^{3}g(k)ve^{kr}(1+e^{kr})}{(ve^{kr} +
    1)^{3}}\Bigl(1+ O\bigl((kt)^{2}\bigr)\Bigr)\\
  &\leq \sum_{k\leq k_{0}} \frac{k^{3}g(k)}{ve^{kr}}\Bigl(1+
    O\bigl((kt)^{2}\bigr)\Bigr)\\
  &\leq \sum_{k\geq 1}\frac{k^{3}g(k)}{ve^{kr}} + O\Biggl(\sum_{k\geq
    1}\frac{k^{5}g(k)t^{2}}{e^{kr}}\Biggr).
\end{align*}
Using the Mellin transform and converse mapping, this results in
\begin{gather*}
  \abs[\bigg]{-i\sum_{k\leq k_{0}}\frac{k^{3}g(k)u^{-1}e^{k(r+it)}(1-u^{-1}e^{k(r+it)})}
  {(u^{-1}e^{k(r+it)}+1)^{3}}}
  \ll r^{-\beta-3}+r^{-\beta-5}t^2.
\end{gather*}
By combining this with~\eqref{eq:sum-over-large-k}, we
obtain
\[\frac{\partial^3 f(r+it, u)}{\partial t^3}\ll r^{-\beta-3}\]
for $\left| t\right|\leq r^{1+3\beta/7}$.

All in all, this leads to the expansion
\begin{equation*}
  f(r + it, u) = f(r, u) - int - \frac{B^{2}}{2}t^{2} + \Oh\bigl(r^{-\beta-3}t^3\bigr).
\end{equation*}
For the integral $I_1$, we thus obtain
\begin{equation}\label{estimate:I_1}
  \begin{split}
  &\frac{e^{nr}}{2\pi}\int_{-t_{n}}^{t_{n}}
  \exp(int+f(r+it, u))\dd{t}\\
  &\quad=\frac{e^{nr+f(r, u)}}{2\pi}\int_{-t_{n}}^{t_{n}}
  \exp\left(-\frac{B^2}{2}t^2+\Oh\bigl(r^{2\beta/7}\bigr)\right)\dd{t}\\
  &\quad=\frac{e^{nr+f(r, u)}}{2\pi}\left(\int_{-t_{n}}^{t_{n}}
  \exp\left(-\frac{B^2}{2}t^2\right)\dd{t}\right)\Bigl(1+\Oh\bigl(r^{2\beta/7}\bigr)\Bigr).
\end{split}
\end{equation}
Finally, we change the integral to a Gaussian integral and get
\begin{equation}
  \label{eq:estimate-I2}
  \begin{split}
  \int_{-t_{n}}^{t_{n}}
    \exp\left(-\frac{B^2}{2}t^2\right)\dd{t}
  &=\int_{-\infty}^{\infty}
    \exp\left(-\frac{B^2}{2}t^2\right)\dd{t}
    -2\int_{t_{n}}^{\infty}
    \exp\left(-\frac{B^2}{2}t^2\right)\dd{t}\\
  &=\frac{\sqrt{2\pi}}{B}+\Oh\biggl(
    \int_{t_{n}}^{\infty}
    \exp\biggl(-\frac{B^2r^{1+3\beta/7}}{2}t\biggr)\dd{t}\biggr)\\
  &=\frac{\sqrt{2\pi}}{B}+\Oh\biggl(
    r^{-1-3\beta/7}B^{-2}\exp\biggl(-\frac{r^{-\beta/7}}{2}\biggr)\biggr).
  \end{split}
\end{equation}
\end{proof}

\subsection{Estimate of $I_{2}$}
Next, we prove the following asymptotic upper bound for the
integral~$I_{2}$.
\begin{lem}
  \label{lem:I2}
  For $I_{2}$, it holds that
  \begin{equation*}
    \abs{I_{2}} \ll \exp\bigl( nr+f(r,u) - c_3r^{-\beta/7}\bigr),
  \end{equation*}
  where $c_{3}$ is a constant uniformly in~$u$.
\end{lem}

For the proof of this estimate, we need the following two lemmas. The first
lemma provides an upper bound for some exponential that will occur later on,
whereas the second one says that $\abs{Q(e^{-r-it},u)}$ is small compared to
$Q(e^{r},u)$. These results are the main ingredients for the proof of
Lemma~\ref{lem:I2}.

\begin{lem}[{Li--Chen~\cite[Lemma~2.5]{li_chen2018:r_th_root}}]
  \label{lc:lem2.5}
  Let $\beta\geq1$. For $\sigma>0$, we have
  \[
    \sum_{k\geq1}k^{\beta-1}e^{-k\sigma}(1-\cos(ky))
    \geq\rho\biggl(\frac{e^{-\sigma}}{(1-e^{-\sigma})^\beta}-
        \frac{e^{-\sigma}}{\abs{1-e^{-\sigma-iy}}^\beta}\biggr),
  \]
  where $\rho$ is a positive constant depending only on $\beta$.
\end{lem}
\begin{lem}\label{lem:Qrit_to_Qr}
  For any real $y$ with $t_{0}\leq\abs{y} \leq \pi$, we derive
  \[\frac{\abs{Q(e^{-(r+iy)},u)}}{Q(e^{-r},u)}
    \leq \exp\biggl(-\frac{2u(\beta-1)}{(1+u)^{2}}\rho\Bigl(\frac14 r^{-\beta/7}\Bigr)\biggr)\]
  for some constant~$\rho$ depending only on~$\beta$.
\end{lem}
\begin{proof}
  First of all, we note that
  \begin{equation*}
    g(k) = \ceil[\big]{(k+1)^{\beta}} - \ceil[\big]{k^{\beta}} \geq (k+1)^{\beta} - k^{\beta} - 1.
  \end{equation*}
By the mean value theorem, there exists $\xi\in(k,k+1)$ such that
\begin{equation*}
  (k+1)^{\beta} - k^{\beta} = \beta\xi^{\beta-1},
\end{equation*}
which leads to
\begin{equation*}
  g(k) \geq \beta\xi^{\beta-1} - 1 = (\beta-1)\xi^{\beta-1} + \xi^{\beta-1} - 1 > (\beta-1)k^{\beta-1};
\end{equation*}
see also Li and Chen~\cite[Proof of Lemma~2.6]{li_chen2018:r_th_root}.
Moreover, we have
\begin{align*}
  \abs{1+ue^{-k(r+iy)}}^{2}
  &=(1+ue^{-kr-kiy})(1+ue^{-kr+kiy})\\
  &=1+ue^{-kr-kiy}+ue^{-kr+kiy}+u^2e^{-2kr}\\
  &=1+2ue^{-kr}+(ue^{-kr})^2-2ue^{-kr}+2ue^{-kr}\cos(ky)\\
  &=(1+ue^{-kr})^2-2ue^{-kr}(1-\cos(ky)).
\end{align*}
Using this, it holds that
\begin{align*}
  \biggl(\frac{\abs{Q(e^{-(r+iy)},u)}}{Q(e^{-r},u)}\biggr)^2
  &=\prod_{k\geq1}\biggl(1-\frac{2ue^{-kr}(1-\cos(ky))}{(1+ue^{-kr})^2}\biggr)^{g(k)}\\
  &=\exp\biggl(\sum_{k\geq1}g(k)\log\biggl(1-\frac{2ue^{-kr}(1-\cos(ky))}{(1+ue^{-kr})^2}\biggr)\biggr)\\
  &\leq\exp\biggl(-\sum_{k\geq1}g(k)\frac{2ue^{-kr}(1-\cos(ky))}{(1+ue^{-kr})^2}\biggr)\\
  &\leq\exp\biggl(-\frac{2u}{(1+u)^2}\sum_{k\geq1}g(k)e^{-kr}(1-\cos(ky))\biggr)\\
  &\leq\exp\biggl(-\frac{2u(\beta-1)}{(1+u)^2}\sum_{k\geq1}k^{\beta-1}e^{-kr}(1-\cos(ky))\biggr)\\
  &\leq\exp\biggl(-\frac{2u(\beta-1)}{(1+u)^{2}}\rho\biggl(\frac{e^{-r}}{(1-e^{-r})^{\beta}}
    - \frac{e^{-r}}{\abs{1-e^{-r-iy}}^{\beta}}\biggr)\biggr),
\end{align*}
where the last estimate follows by Lemma~\ref{lc:lem2.5}.

Following the lines of Li and Chen~\cite{li_chen2018:r_th_root} again,
we further have
\begin{align*}
  \abs{1 - e^{-r - it}}^{\beta} &= (1 - 2e^{-r}\cos t + e^{-2r})^{\beta/2}\\
                                &= ((1 - e^{-r})^{2} +2e^{-r}(1 - \cos t))^{\beta/2}\\
  &\geq (1-e^{-r})^{\beta}\Bigl(1 + \frac{2e^{-r}}{(1 - e^{-r})^{2}}(1 - \cos t_{n})\Bigr)^{\beta/2}
\end{align*}
for $t_{n}\leq \abs{t}\leq \pi$.
Since
\begin{align*}
  1 - \cos t_{n} = \frac12 t_{n}^{2} + \Oh(t_{n}^{4}) = \frac12 r^{2+6\beta/7}+ \Oh(r^{4 + 12\beta/7})
\end{align*}
 and $e^{-r} = 1 - r + \Oh(r^{2})$, we find
\begin{align*}
  \frac{2e^{r}}{(1-e^{r})^{2}}(1 - \cos t_{n}) = \frac{2(1 + \Oh(r))(\frac12 r^{2 + 6\beta/7} + \Oh(r^{4+12\beta/7}))}{r^{2}(1 + \Oh(r))} = r^{6\beta/7}(1 + \Oh(r)).
\end{align*}
This further implies
\begin{align*}
  \abs{1 - e^{-r - it}}^{\beta} &\geq (1 - e^{-r})^{\beta}\bigl(1 + r^{6\beta/7}(1 + \Oh(r))\bigr)^{\beta/2}\\
  &= (1 - e^{r})^{\beta}\Bigl(1 + \frac{\beta}{2}r^{6\beta/7} + \Oh(r^{6\beta/7 + 1})\Bigr).
\end{align*}
This lower bound results in
\begin{align*}
  \frac{e^{r}}{(1-e^{r})^\beta}-\frac{e^{r}}{\left|1-e^{-r-it}\right|^\beta}
  &\geq \frac{e^{r}}{(1-e^{r})^\beta}\left(1-\frac{1}{1+\frac{\beta}2r^{6\beta/7}+\Oh\left(r^{6\beta/7+1}\right)}\right)\\
  &=\frac{e^{r}}{(1-e^{r})^\beta}\left(\frac{\beta}2r^{6\beta/7}+\Oh\left(r^{6\beta/7+1}\right)\right)\\
  &=\frac{1+\Oh(r)}{r^\beta(1-\Oh(r))}\left(\frac{\beta}2r^{6\beta/7}+\Oh\left(r^{6\beta/7+1}\right)\right)\\
  &=\frac{\beta}2 r^{6\beta/7-\beta}+\Oh\bigl(r^{6\beta/7+1-\beta}\bigr)
\end{align*}
and finally
\[
  \frac{e^{r}}{(1-e^{r})^\beta}-\frac{e^{r}}{\left|1-e^{-r-it}\right|^\beta}
  \geq \frac14 r^{6\beta/7-\beta}
\]
for sufficiently large~$n$. So we consequently obtain
\begin{equation*}
  \biggl(\frac{\abs{Q(e^{-(r+iy)}, u)}}{Q(e^{-r}, u)}\biggr)^2 \leq
  \exp\biggl(-\frac{2u(\beta-1)}{(1+u)^{2}}\rho\Bigl(\frac14 r^{6\beta/7-\beta}\Bigr)\biggr),
\end{equation*}
completing the proof.
\end{proof}

\begin{proof}[Proof of Lemma~\ref{lem:I2}]
By the definition of~$I_{2}$, it holds that
\begin{gather*}
  \abs{I_2}
  \leq \frac{e^{nr}}{\pi}\int_{t_n}^\pi \abs{Q(e^{-r-it}, u)}\dd{t}
  = \frac{e^{nr+f(r, u)}}{\pi}\int_{t_n}^\pi \frac{\abs{Q(e^{-r-it}, u)}}{Q(e^{-r},u)}\dd{t}.
\end{gather*}
An application of Lemma~\ref{lem:Qrit_to_Qr} thus yields
\begin{gather}\label{estimate:I_2}
  \abs{I_2} \ll \exp\bigl( nr+f(r,u) - c_3r^{-\beta/7}\bigr)
\end{gather}
for a certain constant $c_3$ uniformly in $u$, as stated.
\end{proof}

\subsection{Estimate of $Q_n(u)$ and Moment-Generating Function}
After estimating the main term and the contribution away from the real axis, we put~\eqref{eq:estimate-I2} and~\eqref{estimate:I_2} together and get
\[
  Q_n(u)
  =\frac{e^{nr+f(u,r)}}{\sqrt{2\pi B}}\bigl(1+\mathcal{O}(r^{2\beta/7})\bigr).
\]

The program for the last part of the proof is to consider the moment-generating
function using this asymptotic expansion. This will prove the central limit
theorem. Finally at the end of this section we use Chernov's bound in order to
obtain the desired tail estimates.


Now we will consider the moment-generating
function for the random variable $\varpi_n$ (the number of summands in
a random partition of $n$). To this end, let
$M_n(t)=\mathbb{E}(e^{(\varpi_n-\mu_n)t/\sigma_n})$, where~$t$ is real
and~$\mu_n$ and~$\sigma_n$ are the mean and the standard deviation
as defined in~\eqref{mu} and~\eqref{sigma}, respectively. Then the following estimation holds.
\begin{lem}
  For bounded~$t$, it holds that
  \begin{equation*}
    M_n(t) = \exp\left(\frac{t^{2}}{2} + \Oh(n^{2\beta/7})\right)
  \end{equation*}
  as $n\to\infty$.
\end{lem}

\begin{proof}
First of all, we observe that
\begin{equation}\label{moment_generating_function}
  \begin{split}
    M_n(t)
    &=\exp\left(-\frac{\mu_nt}{\sigma_n}\right)\frac{Q_n(e^{t/\sigma_n})}{Q_n(1)}\\
    &=\sqrt{\frac{B^2(r(1,n),1)}{B^2(r(n, e^{t/\sigma_n}), e^{t/\sigma_n})}}
      \exp\Bigl(-\frac{\mu_nt}{\sigma_n}+nr(n,e^{t/\sigma_n})\\
    &\qquad +f(r(n, e^{t/\sigma_n}),e^{t/\sigma_n})-nr(n,1)-f(r(n,1),1)+\Oh(r^{2\beta/7})\Bigr).
  \end{split}
\end{equation}
Instead of representing the function with respect to~$u$, we interpret
$r=r(n,u)$ as a function of $n$ and $u$ and use implicit differentiation on
\eqref{n_as_function_of_r} as in
Madritsch and Wagner~\cite{madritsch_wagner2010:central_limit_theorem} and
obtain
\begin{align*}
  r_u=r_u(n,u)=-\frac{f_{u\tau}(r,u)}{f_{\tau\tau}(r,u)}=
  \frac{\sum_{k\geq1}\frac{kg(k)e^{kr}}{(e^{kr}+u)^2}}{\sum_{k\geq1}\frac{uk^2g(k)e^{kr}}{(e^{kr}+u)^2}}
\end{align*}
and similarly
\begin{align*}
  r_{uu}= r_{uu}(n,u) =  \frac{-f_{\tau\tau\tau}(r,u)f_{u\tau}(r,u)^2+2f_{u\tau\tau}(r,u)f_{u\tau}(r,u)f_{\tau\tau}(r,u)-f_{uu\tau}(r,u)f_{\tau\tau}(r,u)^2}{f_{\tau\tau}(r,u)^3}
\end{align*}
as well as
\begin{align*}
r_{uuu}= r_{uuu}(n,u) &= f_{\tau\tau}(r,u)^{-5} \Big(-f_{uuu\tau}(r,u) f_{\tau \tau}(r,u)^4 \\
  &\qquad + \left(3 f_{uu\tau \tau}(r,u) f_{u\tau}(r,u)+3 f_{uu\tau}(r,u)
    f_{u\tau \tau}(r,u)\right) f_{\tau \tau}(r,u)^3 \\
  &\qquad +\left(-3f_{u\tau \tau \tau}(r,u) f_{u\tau}(r,u)^2-6
    f_{u\tau \tau}(r,u)^2 f_{u\tau}(r,u)\right. \\
  &\qquad\qquad \left.-3f_{uu\tau}(r,u) f_{\tau \tau \tau}(r,u) f_{u\tau}(r,u)\right) f_{\tau \tau }(r,u)^2 \\
  &\qquad +\left(f_{\tau \tau \tau \tau}(r,u) f_{u\tau }(r,u)^3+9 f_{u\tau \tau}(r,u)
    f_{\tau \tau \tau }(r,u) f_{u\tau}(r,u)^2\right) f_{\tau \tau
    }(r,u) \\
&\qquad -3 f_{u\tau}(r,u)^3 f_{\tau \tau \tau }(r,u)^2 \Big).
\end{align*}

We now need estimates for the partial derivatives of~$f$. Estimates
for partial derivatives with respect to~$\tau$ follow from our
considerations in Section~\ref{sec:I1}. For derivatives with respect
to~$u$, we take the derivative of the corresponding Mellin transform and
then obtain the estimate via the converse mapping again. Let
us exemplarily illustrate this approach for~$f_{u\tau}$:
By~\eqref{eq:f-tau}, $f_{\tau}$ is given by
\begin{equation*}
  f_{\tau}(\tau, u) = \sum_{\nu=1}^{m}\binom{\beta}{\nu}h_{\beta-\nu, 1}(\tau, u) + \Oh(h_{0, 1}(\tau, u)).
\end{equation*}
The Mellin transform of $h_{\beta - \nu, 1}$ is given by
\begin{equation*}
  H_{\beta - \nu, 1}(s, u) = \zeta(s - \beta + \nu - 1)\Li_{s}(-u)\Gamma(s).
\end{equation*}
Taking the derivative of $H_{\beta - \nu, 1}$ with respect to~$u$ thus
yields
\begin{equation*}
  \pdv{H_{\beta - \nu, 1}}{u}\,\!(s, u) = \frac{1}{u}\zeta(s - \beta + \nu - 1)\Li_{s-1}(-u)\Gamma(s).
\end{equation*}
Consequently, converse mapping implies for $f_{\tau u}$ that
\begin{equation*}
  f_{\tau u}(r, u) = \frac{1}{u}\sum_{\nu=1}^{m}\binom{\beta}{\nu}\Li_{\beta - \nu + 1}(-u)\Gamma(\beta - \nu + 2)r^{-\beta + \nu - 2} + \Oh(r^{-3})\ll r^{-(\beta + 1)}\ll n.
\end{equation*}
In a similar manner, we determine estimates for the other partial derivatives and obtain
\begin{align*}
  f_{u}(r, u) &= \sum_{k\geq 1}\frac{g(k)}{e^{kr} + u} \ll r^{-\beta}\ll n^{\beta/(\beta + 1)},\\
  f_{uu}(r, u) &= -\sum_{k\geq 1}\frac{g(k)}{(e^{kr} + u)^{2}} \ll r^{-\beta}\ll n^{\beta/(\beta + 1)},\\
  f_{\tau\tau}(r,u) &= \sum_{k\geq1} \frac{uk^2g(k)e^{kr}}{(e^{kr}+u)^2} \ll r^{-(\beta+2)}\ll n^{(\beta+2)/(\beta+1)}, \\
  f_{uuu}(r, u) &= \sum_{k\geq 1}\frac{2g(k)}{(e^{kr} + u)^{3}} \ll r^{-\beta}\ll n^{\beta/(\beta + 1)},\\
  f_{uu\tau}(r,u) &= \sum_{k\geq1} \frac{2kg(k)e^{kr}}{(e^{kr}+u)^3} \ll r^{-(\beta +1)}\ll n,\\
  f_{u\tau\tau}(r,u) &= \sum_{k\geq1} \frac{k^2g(k)e^{kr}(e^{kr}-u)}{(e^{kr}+u)^3} \ll r^{-(\beta +2)}\ll n^{(\beta +2)/(\beta +1)}, \\
  f_{\tau\tau\tau}(r,u) &= - \sum_{k\geq1} \frac{uk^3g(k)e^{kr}(e^{kr}-u)}{(e^{kr}+u)^3} \ll r^{-(\beta +3)}\ll n^{(\beta +3)/(\beta +1)}, \\
  f_{uuu\tau}(r,u) &= - \sum_{k\geq1} \frac{6kg(k)e^{kr}}{(e^{kr}+u)^4} \ll r^{-(\beta +1)}\ll n, \\
  f_{uu\tau\tau}(r,u) &= - \sum_{k\geq1} \frac{2k^2g(k)e^{kr}(2e^{kr}-u)}{(e^{kr}+u)^4} \ll r^{-(\beta +2)}\ll n^{(\beta +2)/(\beta +1)}, \\
  f_{u\tau\tau\tau}(r,u) &= - \sum_{k\geq1} \frac{k^3g(k)e^{kr}(e^{2kr}-4ue^{kr}+u^2)}{(e^{kr}+u)^4} \ll r^{-(\beta +3)} \ll n^{(\beta +3)/(\beta +1)}, \\
  f_{\tau\tau\tau\tau}(r,u) &= \sum_{k\geq1} \frac{uk^4g(k)e^{kr}(e^{2kr}-4ue^{kr}+u^2)}{(e^{kr}+u)^4} \ll r^{-(\beta +4)} \ll n^{(\beta +4)/(\beta +1)}.
\end{align*}
From these estimates it follows that $r_u,r_{uu},r_{uuu}\ll n^{-1/(\beta + 1)}$
uniformly in $u$. Expanding $r(n,e^{t/\sigma_n})$ and
$f(r(e^{t/\sigma_n},n),e^{t/\sigma_n})$ around $t=0$ yields
\[
  r(n,e^{t/\sigma_n}) - r(n,1)
  =r_u(n,1)\cdot \frac{t}{\sigma_n}
  +\frac{r_u(n,1)+r_{uu}(n,1)}{2}\cdot\left(\frac{t}{\sigma_n}\right)^2
  +\Oh\left(n^{-1/(\beta +1)}\frac{t^3}{\sigma_n^3}\right)
\]
and
\begin{align*}
  &f(r(e^{t/\sigma_n},n)e^{t/\sigma_n}) - f(r(1,n),1)
  =\left(f_{\tau}(r(n,1),1)r_u(n,1)+f_u(r(n,1),1)\right)\cdot\frac{t}{\sigma_n}\\
  &\quad+\frac12\left(f_{\tau}(r(n,1),1)(r_u(n,1) + r_{uu}(n,1))
    +f_{\tau\tau}(r(n,1),1)r_u(n,1)^2+2f_{u\tau}(r(n,1),1)r_u(n,1)\right.\\
  &\quad\quad+\left.f_u(r(n,1),1)+f_{uu}(r(n,1),1)\right)
    \cdot\left(\frac{t}{\sigma_n}\right)^2
    +\Oh\left(n^{\beta/(\beta +1)}\frac{t^3}{\sigma_n^3}\right),
\end{align*}
respectively.

By plugging these expansions into the exponential of
\eqref{moment_generating_function} we get
\begin{align*}
  &\bigl(nr_u(n,1)+f_\tau(\eta,1)r_u(n,1)+f_u(\eta,1)-\mu_n\bigr)
    \cdot\frac{t}{\sigma_n}\\
  &\quad+\frac12\bigl(n(r_u(n,1)+r_{uu}(n,1)\bigr)+f_{\tau}(\eta,1)\bigl(r_u(n,1)+r_{uu}(n,1)\bigr)
    +f_{\tau\tau}(\eta,1)r_u(n,1)^2\\
  &\quad\quad + 2f_{u\tau}(\eta,1)r_{u}(n,1) + f_{u}(\eta, 1) + f_{uu}(\eta, 1)\bigr)
    \cdot\left(\frac{t}{\sigma_n}\right)^2
    +\Oh\left(n^{\beta/(\beta + 1)}\frac{t^3}{\sigma_n^3}+n^{2\beta/7}\right),
\end{align*}
where we wrote $\eta=r(n,1)$ for short. Recalling that $n=-f_\tau(\eta,1)$ and
that $r_u(n,1)=-\frac{f_{u\tau}(\eta,1)}{f_{\tau\tau}(\eta,1)}$, we can simplify
the last expression in order to obtain
\begin{multline*}
  \left(f_u(\eta,1)-\mu_n\right)\cdot\frac{t}{\sigma_n}
  +\frac12\left(f_u(\eta,1)+f_{uu}(\eta,1)
  -\frac{f_{u\tau}(\eta,1)^2}{f_{\tau\tau}(\eta,1)}\right)
  \cdot\left(\frac{t}{\sigma_n}\right)^2
  +\Oh\biggl(n^{\beta/(\beta + 1)}\frac{t^3}{\sigma_n^3}+n^{2\beta/7}\biggr).
\end{multline*}
In a similar way we get that
\[\frac{B^2(r(n,1),1)}{B^2(r(n,e^{t/\sigma_n}), e^{t/\sigma_n})}
  =1+\Oh\left(\frac{t}{\sigma_n}\right).\]
Thus, we obtain the following asymptotic formula for the moment-generating
function in~\eqref{moment_generating_function}:
\begin{multline*}
  M_n(t)=\exp\biggl(\bigl(f_u(\eta,1)-\mu_n\bigr)\cdot\frac{t}{\sigma_n}
    +\frac12\left(f_u(\eta,1)+f_{uu}(\eta,1)
    -\frac{f_{u\tau}(\eta,1)^2}{f_{\tau\tau}(\eta,1)}\right)
    \cdot\left(\frac{t}{\sigma_n}\right)^2\\
    +\Oh\biggl(\frac{t}{\sigma_n}+n^{\frac{\beta}{\beta + 1}}\frac{t^3}{\sigma_n^3}+n^{-\beta + 9\beta/7}\biggr)\biggr).
\end{multline*}
Thus, our choice of $\mu_{n}$ and $\sigma_{n}$ in~\eqref{mu} and~\eqref{sigma} is justified. Mellin transform technique finally shows that
\begin{equation}\label{eq:mu-explicit}
  \mu_n
      \sim -\beta\Li_{\beta}(-1)\Gamma(\beta)\bigl(-\beta\Li_{\beta+1}(-1)\Gamma(\beta+1)\bigr)^{-\beta/(\beta+1)}n^{\beta/(\beta + 1)}
\end{equation}
and
\begin{multline}\label{eq:sigma-explicit}
  \sigma_n^2 \sim \mu_{n} + \Biggl(\frac{-\beta\Li_{\beta-1}(-1)\Gamma(\beta)}{\bigl(-\beta\Li_{\beta+1}(-1)\Gamma(\beta+1)\bigr)^{\beta/(\beta+1)}}\\
  + \frac{(\beta\Li_{\beta}(-1)\Gamma(\beta + 1)\bigr)^{2}}{\beta\Li_{\beta+1}(-1)\Gamma(\beta+2)\bigl(-\beta\Li_{\beta+1}(-1)\Gamma(\beta+1)\bigr)^{\beta/(\beta+1)}}\Biggr)n^{\beta/(\beta + 1)}.
\end{multline}
Plugging everything into the asymptotic formula for the moment-generating function
we have
\begin{align*}
  M_n(t)&=\exp\left(\frac{t^2}{2}+\Oh\bigl(n^{-\beta/(2\beta + 2)}t + n^{-\beta/(2\beta + 2)}t^{3} + n^{2\beta/7}\bigr)\right)\\
        &= \exp\left(\frac{t^2}{2}+\Oh\bigl(n^{-\beta/(2\beta + 2)}(t + t^{3}) + n^{2\beta/7}\bigr)\right)\\
  &= \exp\left(\frac{t^{2}}{2} + \Oh(n^{2\beta/7})\right)
\end{align*}
for bounded~$t$.
\end{proof}

By the previous lemma and Curtiss' theorem~\cite{curtiss_1942:moment_generating_functions}, it follows that the
distribution of $\varpi_{n}$ is indeed asymptotically normal.
For the remaining parts, we again follow Hwang~\cite{hwang2001:limit_theorems_number} and obtain for $t = o(n^{\alpha/(6\alpha +6)})$ that 
\begin{align*}
  \mathbb{P}\left(\frac{\varpi - \mu_n}{\sigma_n}\geq x\right)
  &\leq e^{-tx}M_n(t)\\
  &= e^{-tx+t^2/2}\left(1+\Oh\bigl(n^{-\beta/(2\beta + 2)}(t + t^{3}) + n^{2\beta/7}\bigr)\right)
\end{align*}
by Chernov's inequality.

Finally, let $T=n^{\beta/(6\beta+6)}$. Then for $x\leq T$ we set $t=x$ and obtain
\begin{align*}
  \mathbb{P}\left(\frac{\varpi - \mu_n}{\sigma_n}\geq x\right)
  \leq e^{-x^2/2}\left(1+\Oh\left((\log n)^{-3}\right)\right).
\end{align*}
For $x\geq T$ we set $t=T$ yielding
\begin{align*}
  \mathbb{P}\left(\frac{\varpi - \mu_n}{\sigma_n}\geq x\right)
  \leq e^{-Tx/2}\left(1+\Oh\left((\log n)^{-3}\right)\right).
\end{align*}
We can estimate the probability $\mathbb{P}\bigl(\frac{\varpi -
\mu_n}{\sigma_n}\leq -x\bigr)$ in a similar way.

\section*{Acknowledgment}
The first author is supported by the Austrian Science Fund (FWF), project
W\,1230. The second author is supported by project ANR-18-CE40-0018 funded by
the French National Research Agency. The third author is supported by the
Austrian Science Fund (FWF), project F\,5510-N26 within the Special Research
Area “Quasi-Monte Carlo Methods: Theory and Applications” and project
I\,4406-N.

Major parts of the present paper were established when the first
author was visiting the Institut Élie Cartan at the
Université de Lorraine, France. He thanks the institution for
its hospitality.

Finally, the authors thank the reviewer for carefully
reading the manuscript and for the helpful suggestions. Her/his valuable
comments improved the quality of the article.

\bibliography{literatur}

\end{document}